\documentclass[11pt]{amsart}

\usepackage{amsmath,xspace,amssymb,mathrsfs}
\usepackage{color}

\input xy
\xyoption{all}
\xyoption{2cell}
\UseAllTwocells
\CompileMatrices

\newcommand{\Spec}{\operatorname{Spec}}
\renewcommand{\phi}{\varphi}

\newcommand{\Ass}{\operatorname{Ass}}

\newcommand{\depth}{\operatorname{depth}}

\newcommand{\Max}{\operatorname{Max}}

\newcommand{\Min}{\operatorname{Min}}
\newcommand{\Ann}{\operatorname{Ann}}
\newcommand{\Hom}{\operatorname{Hom}}

\newcommand{\Supp}{\operatorname{Supp}}

\newcommand{\Fin}{\operatorname{Fin}}

\newtheorem{proposition}{Proposition}[section]
\newtheorem{lemma}[proposition]{Lemma}

\newtheorem{corollary}[proposition]{Corollary}
\newtheorem{theorem}[proposition]{Theorem}

\theoremstyle{definition}
\newtheorem{definition}[proposition]{Definition}

\newtheorem{remark}[proposition]{Remark}

\usepackage{etoolbox}
\makeatletter
\patchcmd{\@settitle}{\uppercasenonmath\@title}{}{}{}
\patchcmd{\@setauthors}{\MakeUppercase}{}{}{}
\makeatother

\begin{document}

\title[Cardinality of groups and rings]{Cardinality of groups and rings via the idempotency of infinite cardinals}

\author[A. Tarizadeh]{Abolfazl Tarizadeh}
\address{Department of Mathematics, Faculty of Basic Sciences, University of Maragheh, Maragheh, East Azerbaijan Province, Iran.}
\email{ebulfez1978@gmail.com}

\date{}

\begin{abstract}  An important classical result in ZFC asserts that every infinite cardinal number is idempotent. Using this fact, we obtain several algebraic results in this article. The first result  asserts that an infinite Abelian group has a proper subgroup with the same cardinality if and only if it is not a Prüfer group. In the second result, the cardinality of any monoid-ring $R[M]$ (not necessarily commutative) is calculated. In particular, the cardinality of every polynomial ring with any number of variables (possibly infinite) is easily computed. Next, it is shown that every commutative ring and its total ring of fractions have the same cardinality. This set-theoretic observation leads us to a notion in ring theory that we call  a balanced ring (i.e. a ring that is canonically isomorphic to its total ring of fractions). Every zero-dimensional ring is a balanced ring. Then we show that a Noetherian ring is a balanced ring if and only if its localization at every maximal ideal has zero depth. It is also proved that every self-injective ring (injective as a module over itself) is a balanced ring.   
\end{abstract}

\maketitle

\section{Introduction}

In Zermelo-Fraenkel (ZF) set theory, the axiom of choice is equivalent to the statement that every infinite cardinal number is idempotent. The fact that the choice yields the idempotency of infinite cardinals is quite well known (see e.g. \cite[Lemma 6R]{Enderton} or \cite[Theorem 8.7]{Math Garden}). The reverse implication is less known in the literature and was proved by Alfred Tarski (published in 1924). It is worth mentioning that Tarski proved the above result earlier than that date. But two famous mathematicians (Fréchet and Lebesgue) who were probably the reviewers of Tarski's article by giving strange reasons did not allow him to publish his result in Comptes rendus de l'Académie des Sciences. Fréchet wrote that an implication between two well known propositions is not a new result. Lebesgue wrote that an implication between two false propositions is of no interest. Tarski then published his result in Fundamenta Mathematicae \cite{Tarski}. Note that the idempotency of the least infinite cardinal $\aleph_{0}$ (the cardinality of natural numbers) can be proved in ZF without using the axiom of choice. This interesting result was proved by Georg Cantor in the last years of 19th century (in fact, it was the primary motivation for discovering the idempotency of all infinite cardinals). Indeed, if $\omega=\{0,1,2, \ldots\}$ denotes the set of natural numbers, then it can be seen that the map $\omega\times\omega\rightarrow\omega$ given by $(m,n)\mapsto m+\sum\limits_{k=1}^{m+n}k=m+\frac{(m+n)(m+n+1)}{2}$ is injective (even bijective) and hence $\aleph_{0}=\aleph_{0}^{2}$. On the other hand, in 1949, Szélpál published a short two-page article \cite{Szelpal}. The only result of this article was an interesting theorem in group theory which asserts that  the Prüfer groups are the only infinite Abelian groups with only finite proper subgroups. Next in 1986, Gilmer and Heinzer \cite[Theorem 1.2]{Gilmer-Heinzer} proved a technical result in the realm of commutative algebra which asserts that every uncountable commutative ring has a proper subring with the same cardinality. In this article, by applying these two major results and the idempotency of infinite cardinals, we obtain the following characterization: 

\begin{theorem} An infinite Abelian group has a proper subgroup with the same cardinality if and only if it is not a Prüfer group. 
\end{theorem}

It is worth noting that the above result is not true in the non-commutative case. In fact, Shelah \cite[Theorem A]{Shelah} discovered a surprising example of  a (non-Abelian) uncountable group of cardinality $\aleph_{1}=2^{\aleph_{0}}$ such that every proper subgroup is countable (i.e. its cardinality $\leqslant\aleph_{0}$). To the best of the author's knowledge, it is not yet known whether the Gilmer-Heinzer theorem \cite[Theorem 1.2]{Gilmer-Heinzer} is true or false for non-commutative rings. But according to Shelah's counterexample, it is also very likely Gilmer-Heinzer theorem to be false in non-commutative cases.

The idempotency of infinite cardinals also leads us to the following general result (in ZFC):

\begin{theorem} For any nonzero ring $R$ and any monoid $M$ (which are not necessarily commutative), the cardinality of the monoid-ring $R[M]$ is $|R|^{|M|}$ if $M$ is finite, otherwise it is $|R|\times|M|$. 
\end{theorem} 

As a consequence of the above result, if $R$ is a nonzero ring and $S$ is an index set with $|S|\geqslant1 $ then the cardinality of the polynomial ring $R[x_k: k\in S]$ is  $\aleph_{0}\times|R|\times|S|$.

We also show that every commutative ring and its total ring of fractions have the same cardinality. This result shows that every commutative ring is equinumerous (in bijection) with its total ring of fractions. Inspired by this observation, it is natural to ask for which rings $R$ the canonical ring map $R\rightarrow T(R)$ is an isomorphism of rings where $T(R)$ denotes the total ring of fractions of $R$. Every ring $R$ that has this property is called a \emph{ balanced ring}. Next, we investigate the classes of rings that are balanced rings. In particular, we show that every zero dimensional ring is a balanced ring. Hence all of the fields, finite rings, Artinian rings, Boolean rings and more generally von-Neumann regular (absolutely flat) rings are balanced rings. We prove that a Noetherian ring is a balanced ring if and only if its localization at every maximal ideal has zero depth. We also show that every self-injective ring is a balanced ring.

\section{Preliminaries}

Throughout this article, the cardinality of a set $S$ is denoted by $|S|$. For given sets $A$ and $B$, by $B^{A}$ we mean the set of all functions $A\rightarrow B$ from $A$ to $B$. The idempotency of infinite cardinals asserts that if $\kappa$ is an infinite cardinal number then $\kappa=\kappa^{2}$.

\begin{remark}\label{Remark 3-III} The idempotency of infinite cardinals yields that if $\alpha$ is an infinite cardinal number and $\beta$ is any cardinal number then the following assertions hold: \\
$\mathbf{(i)}$ $\alpha+\beta=\max\{\alpha,\beta\}$. \\
$\mathbf{(ii)}$ $\alpha\cdot\beta=\max\{\alpha,\beta\}$ provided that $\beta\neq0$. \\
$\mathbf{(iii)}$ If $2\leqslant\beta\leqslant\alpha$ then $\beta^{\alpha}=2^{\alpha}$. \\
$\mathbf{(iv)}$ If $\alpha<\beta$ then $\beta^{\alpha}\in\{\beta,2^{\beta}\}$.  
\end{remark}

In particular, if $S$ is an infinite set and $A$ is a finite set then $|S\setminus A|=|S|=|S\cup A|$. 

Let $\{A_{k}: k\in S\}$ be a family of sets. Recall that the cardinality of the disjoint union set $\coprod\limits_{k\in S}A_{k}=\bigcup\limits_{k\in S}A_{k}\times\{k\}$ is often denoted by the notation $\sum\limits_{k\in S}|A_{k}|$.  The cardinality of the direct product set $\prod\limits_{k\in S}A_{k}$ is also denoted by the notation $\prod\limits_{k\in S}|A_{k}|$. If the $A_{k}$ with $k\in S$ are pairwise disjoint then the cardinality of $\bigcup\limits_{k\in S}A_{k}$ equals $\sum\limits_{k\in S}|A_{k}|$.

\begin{remark}\label{Remark 2} Let $M=(x_{1},\ldots,x_{n})$ be a finitely generated module over a ring $R$. If $R/\Ann(x_{k})$ is a finite ring for all $k$, then $M$ is a finite set. Indeed, the map $\bigoplus\limits_{k=1}^{n}R/\Ann(x_{k})\rightarrow M$ given by $\big(r_{k}+\Ann(x_{k})\big)\mapsto
\sum\limits_{k=1}^{n}r_{k}x_{k}$ is a surjective morphism of $R$-modules. 
\end{remark}

\section{Cardinality of Groups}

It is obvious that a finite algebraic structure (such as a group, ring etc) has no proper subalgebra of the same cardinality (since a finite set has no proper subset of the same cardinality). But in the infinite case, different things can happen for algebraic structures which is the subject of the following results.     

\begin{lemma}\label{Theorem 3-uch} An infinite Abelian group $G$ has a proper subgroup with the same cardinality if one of the following conditions hold: \\
$\mathbf{(i)}$ $G$ is uncountable. \\
$\mathbf{(ii)}$ $G$ is countable and finitely generated.
\end{lemma}

\begin{proof} $\mathbf{(i):}$ We know that every Abelian group naturally has a $\mathbb{Z}$-module structure. Then consider the Nagata idealization ring $\mathbb{Z}\ltimes G:=\mathbb{Z}\times G$ whose addition is the componentwise $(r,x)+(r',y)=(r+r',x+y)$ and whose multiplication is defined as $(r,x)\cdot(r',y)=(rr',ry+r'x)$. Using Remark \ref{Remark 3-III}(ii), then it is clear that this ring $\mathbb{Z}\ltimes G$ is an uncountable commutative ring of cardinality $|\mathbb{Z}|\times |G|=|G|$. Then by applying  \cite[Theorem 1.2]{Gilmer-Heinzer} which asserts that every uncountable commutative ring has a proper subring of the same cardinality, we get a proper subring $R$ of $\mathbb{Z}\ltimes G$ of cardinality $|G|$. But the unit element $(1,0)$ of the ring $\mathbb{Z}\ltimes G$ is a member of $R$ where $0$ is the identity element of $G$. It follows that $\mathbb{Z}\times\{0\}\subseteq R$. The map $\mathbb{Z}\rightarrow\mathbb{Z}\ltimes G$ given by $r\mapsto(r,0)$ is a morphism of rings and so $\mathbb{Z}\ltimes G$ has a $\mathbb{Z}$-module structure through this map. Then the map $f:G\rightarrow\mathbb{Z}\times G$ given by $x\mapsto(0,x)$ is a morphism of $\mathbb{Z}$-modules (Abelian groups) and so $H=f^{-1}(R)$ is a subgroup of $G$. But we have $R=\mathbb{Z}\times H$. Thus $H$ is a ``proper" subgroup of $G$. We also have  $|G|=|R|=|\mathbb{Z}\times H|=\max\{\aleph_{0},|H|\}=|H|$. \\
$\mathbf{(ii):}$ By the structure theorem of finitely generated Abelian groups or more generally by the structure theorem of finitely generated modules over PIDs (principal ideal domains), we have a factorization $G=T(G)\oplus F$ where $T(G)=\{x\in G:\exists n\geqslant1, nx=0\}$ is the trosion subgroup of $G$ (the subgroup of all elements of finite order) and 
$F$ is a free subgroup of $G$ of finite rank (i.e. $F\simeq \mathbb{Z}^{d}$ for some natural number $d\geqslant0$). But $F$ is nonzero, because otherwise $G=T(G)$ is generated by finitely many elements of finite order and hence by Remark \ref{Remark 2}, $G$ will be a finite group which is a contradiction. Then $F\simeq \mathbb{Z}^{d}$ for some positive natural number $d\geqslant1$. But $\mathbb{Z}^{d}$ has a proper (free) subgroup of infinite order (it is clear for $d\geqslant2$, if $d=1$ then $2\mathbb{Z}$ is an infinite proper subgroup of $\mathbb{Z}$). Thus $G$ has a proper subgroup $H$ of infinite order. But $G$ is countable and so $|H|=|G|$.
\end{proof}

It is important to notice that part (ii) of the above lemma fails in the infinitely generated case.   
As an example, fix a prime number $p$ then the Prüfer group $\mathbb{Z}(p^{\infty})=\mathbb{Z}[1/p]/\mathbb{Z}$ is countably infinite Abelian group but it has no proper subgroup of infinite order. In fact, every proper subgroup of the  
Prüfer group $\mathbb{Z}(p^{\infty})$ is precisely of the form $(1/p^{n}\mathbb{Z})/\mathbb{Z}$ for some natural number $n\geqslant0$ which is a finite group of order $p^{n}$. Also note that Lemma \ref{Theorem 3-uch} is not true in the non-commutative case. Indeed by \cite[Theorem A]{Shelah}, there exists a non-Abelian uncountable group of cardinality $\aleph_{1}=2^{\aleph_{0}}$ such that every proper subgroup is countable (i.e. its cardinality $\leqslant\aleph_{0}$).

After proving Lemma \ref{Theorem 3-uch}, we noticed that similar results have already been proved by Scott \cite[Theorem 9]{Scott} and Szélpál \cite{Szelpal}. By combining this lemma with Szélpál's theorem, we obtain the following more general result: 

\begin{theorem}\label{Theorem 4-iv} An Abelian group $G$ has no proper subgroup with the same cardinality if and only if $G$ is a finite group or a Prüfer group.
\end{theorem}

\begin{proof} First assume $G$ has no proper subgroup with the same cardinality. If $G$ is not a finite group, then by Lemma \ref{Theorem 3-uch}, $G$ is a non-finitely generated countably infinite Abelian group. Hence, every proper subgroup of $G$ is finite. 
But by Szélpál's theorem \cite{Szelpal} which asserts that the Prüfer groups are the only infinite Abelian groups with only finite proper subgroups, we get that $G$ is a Prüfer group, i.e. $G\simeq\mathbb{Z}(p^{\infty})$ for some prime number $p$. The reverse implication is clear, since a Prüfer group has no proper subgroup of infinite order.
\end{proof}

Under the light of the above result, then Lemma \ref{Theorem 3-uch} can be improved as:

\begin{corollary} An infinite Abelian group has a proper subgroup with the same cardinality if and only if it is not a Prüfer group.
\end{corollary}

\begin{proof} It is an immediate consequence of Theorem \ref{Theorem 4-iv}.
\end{proof}

\section{Cardinality of rings}

To prove the main results of this section, we need the following lemmas which are interesting in their own way.

\begin{lemma}\label{Lemma 1} For any sets $M$ and $S$ we have $\sum\limits_{k\in S}|M|=|M|\times|S|$ and $\prod\limits_{k\in S}|M|=|M|^{|S|}$.
\end{lemma}

\begin{proof} It is clear that $\coprod\limits_{k\in S}M=M\times S$. It follows that $\sum\limits_{k\in S}|M|=|M\times S|=|M|\times|S|$. We also have $\prod\limits_{k\in S}M=M^{S}$ and so $\prod\limits_{k\in S}|M|=|M|^{|S|}$. 
\end{proof}

\begin{remark}\label{Remark 1} If $\{A_{k}: k\in S\}$ and $\{B_{k}: k\in S\}$ are two families of sets such that $|A_{k}|\leqslant|B_k|$ for all $k$, then $\sum\limits_{k\in S}|A_{k}|\leqslant\sum\limits_{k\in S}|B_{k}|$ and $\prod\limits_{k\in S}|A_{k}|\leqslant\prod\limits_{k\in S}|B_{k}|$. Indeed by the hypothesis, for each $k\in S$ we have an injective function $f_k:A_{k}\rightarrow B_{k}$ and so the maps $\coprod\limits_{k\in S}A_{k}\rightarrow\coprod\limits_{k\in S}B_{k}$ and $\prod\limits_{k\in S}A_{k}\rightarrow\prod\limits_{k\in S}B_{k}$ respectively given by $(x,k)\mapsto\big(f_{k}(x),k\big)$ and $(x_{k})\mapsto\big(f_{k}(x_{k})\big)$ are injective. It is also well known that if $|A_{k}|<|B_k|$ for all $k\in S$, then $\sum\limits_{k\in S}|A_k|<\prod\limits_{k\in S}|B_k|$.
\end{remark}

For any set $S$ by $\Fin(S)$ we mean the set of all finite subsets of $S$. If $S$ is finite then $\Fin(S)=\mathscr{P}(S)$ and hence $|\Fin(S)|=2^{|S|}$.
In the infinite case we have the following result:

\begin{lemma}\label{Lemma 4} A set $S$ is infinite if and only if $|\Fin(S)|=|S|$.
\end{lemma}

\begin{proof} First assume $S$ is infinite. The map $S\rightarrow\Fin(S)$ given by $x\mapsto\{x\}$ is injective and hence $|S|\leqslant|\Fin(S)|$. To prove the reverse inequality we act as follows. 
For each natural number $n\geqslant0$, let $F_{n}(S)$ be the set of all (finite) subsets of $S$ of cardinality $n$. We have $\Fin(S)=\bigcup\limits_{n\geqslant0}F_{n}(S)$ and the $F_{n}(S)$ are pairwise disjoint. It follows that $|\Fin(S)|=\sum\limits_{n\geqslant0}|F_{n}(S)|$. 
For each $n\geqslant1$, the map $S^{n}\rightarrow\bigcup\limits_{d=1}^{n}F_{d}(S)$ given by $(x_{1},\ldots,x_{n})\mapsto\{x_{1},\ldots,x_{n}\}$ is surjective. Using the idempotency of infinite cardinals, it follows that $|F_{n}(S)|\leqslant\beta_{n}:=
\sum\limits_{d=1}^{n}|F_{d}(S)|
\leqslant|S|^{n}=|S|$ for all $n\geqslant1$. Then by Remark \ref{Remark 1}, $|\Fin(S)|=1+\sum\limits_{n\geqslant1}
|F_{n}(S)|\leqslant1+\sum\limits_{n\geqslant1}\beta_{n}
\leqslant1+\sum\limits_{n\geqslant1}|S|=
1+\aleph_{0}\times|S|=|S|$. Thus $|\Fin(S)|=|S|$. The reverse implication follows from the Cantor theorem (which asserts that for any set $S$ then $|S|<2^{|S|}$).
\end{proof}

\begin{lemma}\label{Lemma 2} Let $\{A_{k}: k\in S\}$ be a family of sets. If the index set $S$ is infinite and $1\leqslant|A_{k}|\leqslant|S|$ for all $k$, then $\sum\limits_{k\in S}|A_{k}|=|S|$.
\end{lemma}

\begin{proof} By Lemma \ref{Lemma 1} and Remark \ref{Remark 1} and using the idempotency of infinite cardinals, we have $|S|=\sum\limits_{k\in S}1\leqslant\sum\limits_{k\in S}|A_{k}|\leqslant\sum\limits_{k\in S}|S|=|S|^{2}=|S|$ and so $\sum\limits_{k\in S}|A_{k}|=|S|$.
\end{proof}

\begin{lemma} Let $\{A_{k}: k\in S\}$ be a family of sets. If the index set $S$ is infinite and $2\leqslant|A_{k}|\leqslant|S|$ for all $k$, then $\prod\limits_{k\in S}|A_{k}|=2^{|S|}$.
\end{lemma}

\begin{proof} By Lemma \ref{Lemma 1} and Remarks \ref{Remark 3-III} and \ref{Remark 1}, we have $2^{|S|}=\prod\limits_{k\in S}2\leqslant\prod\limits_{k\in S}|A_{k}|\leqslant\prod\limits_{k\in S}|S|=|S|^{|S|}=2^{|S|}$. Thus $\prod\limits_{k\in S}|A_{k}|=2^{|S|}$. 
\end{proof}

\begin{remark} Let $\{R_{k}: k\in S\}$ be a family of nonzero rings. If the index set $S$ is infinite and $|R_k|\leqslant|S|$ for all $k$, then the cardinality of the direct product ring $\prod\limits_{k\in S}R_{k}$ is $2^{|S|}$. Similarly, if $R$ is a finite nonzero ring then the cardinality of the formal power series ring $R[[x]]$ is $2^{\aleph_{0}}$. If $R$ is an infinite ring of cardinality $\kappa$ then the cardinality of  $R[[x]]$ is $\kappa^{\aleph_{0}}$. For example, the cardinality of $\mathbb{C}[[x]]$ is $(2^{\aleph_{0}})^{\aleph_{0}}=2^{\aleph_{0}}$. As another example, if $R$ is a ring of cardinality $\aleph_{\omega}=\sup\{\aleph_{d}: d\geqslant0\}$ then $R[[x]]$ is of cardinality $(\aleph_{\omega})^{\aleph_{0}}=2^{\aleph_{\omega}}$, because $2^{\aleph_{\omega}}\geqslant
(\aleph_{\omega})^{\aleph_{0}}=
\prod\limits_{d\geqslant0}\aleph_{\omega}>
\sum\limits_{d\geqslant0}\aleph_{d}=\aleph_{\omega}$ and hence $(\aleph_{\omega})^{\aleph_{0}}=2^{\aleph_{\omega}}$. Note that there are infinite rings of any cardinality. More precisely, for any infinite cardinal $\kappa$, there exists a (Boolean) ring $R$ with $|R|=\kappa=|\Spec(R)|$. 
\end{remark}

Let $S$ be an index set, $M$ a nonempty set and let $e$ be a (special) fixed element in $M$. Then by $\bigoplus\limits_{k\in S}M$ we mean the set of all functions $f:S\rightarrow M$ with finite support (with respect to $e$), i.e. $\Supp(f)=\{k\in S: f(k)\neq e\}$ is a finite set. In fact,  $\bigoplus\limits_{k\in S}M$ is the set of all sequences $(x_{k})_{k\in S}$ in the direct product set $\prod\limits_{k\in S}M$ such that the $x_{k}=e$ except for a finite number of indices $k$. It is clear that the cardinality of $\bigoplus\limits_{k\in S}M$ is independent of the choice of $e\in M$.  For example, if $M$ is a monoid then $e$ is the identity element of $M$.
If $M$ is the empty set then $\bigoplus\limits_{k\in S}M$ is also considered the empty set.

\begin{lemma}\label{Lemma 3} For any sets $M$ and $S$, we have $|\bigoplus\limits_{k\in S}M|=|M|^{|S|}$ if $S$ is  finite or $|M|\leqslant1$, otherwise $|\bigoplus\limits_{k\in S}M|=|M|\times |S|$.
\end{lemma}

\begin{proof} If $S$ is finite then $\bigoplus\limits_{k\in S}M=\prod\limits_{k\in S}M$ and so $|\bigoplus\limits_{k\in S}M|=|M|^{|S|}$. If $M$ is singleton then $\bigoplus\limits_{k\in S}M$ is also singleton and so $|\bigoplus\limits_{k\in S}M|=1=|M|^{|S|}$. If $|M|=0$ then $|\bigoplus\limits_{k\in S}M|=0=|M|^{|S|}$. To prove the rest of the assertion we proceed as follows.  
It can be seen that the map $\bigoplus\limits_{k\in S}M\rightarrow\bigcup
\limits_{A\in\Fin(S)}(M\setminus\{e\})^{A}$ which sends each $(x_k)\in\bigoplus\limits_{k\in S}M$ into the function $f:\{k\in S: x_{k}\neq e\}\rightarrow M\setminus\{e\}$ given by $f(k)=x_k$ is bijective. The sets $(M\setminus\{e\})^{A}$ with $A\in\Fin(S)$ are pairwise disjoint. Hence,  $|\bigoplus\limits_{k\in S}M|=
\sum\limits_{A\in\Fin(S)}(|M|-1)^{|A|}$. Now if $M$ is infinite then using Remark \ref{Remark 3-III}, we have $(|M|-1)^{|A|}=|M|$ for all $A\in\Fin(S)$. Then by using Lemma \ref{Lemma 1}, we will have $|\bigoplus\limits_{k\in S}M|=\sum\limits_{A\in\Fin(S)}|M|=|M|\times|\Fin(S)|
=|M|\times|S|$ (note that if $S$ is finite then $|M|\times|\Fin(S)|=|M|=|M|\times|S|$, but if $S$ is infinite then by Lemma \ref{Lemma 4}, $|\Fin(S)|=|S|$). Finally, assume $M$ is finite with $|M|\geqslant2$. We know that $S$ is infinite. Thus $1\leqslant(|M|-1)^{|A|}<\aleph_{0}\leqslant|S|=
|\Fin(S)|$ for all $A\in\Fin(S)$.
Then by Lemmas \ref{Lemma 4} and \ref{Lemma 2}, $|\bigoplus\limits_{k\in S}M|=\sum\limits_{A\in\Fin(S)}(|M|-1)^{|A|}=|\Fin(S)|=
|S|=|M|\times|S|$.   
\end{proof}

\begin{theorem}\label{Theorem 1} For any nonzero ring $R$ and any monoid $M$, the cardinality of the monoid-ring $R[M]$ is $|R|^{|M|}$ if $M$ is finite otherwise equal to $|R|\times|M|$. 
\end{theorem}

\begin{proof} We know that $R[M]=\bigoplus\limits_{x\in M}R$. Then apply Lemma \ref{Lemma 3}.
\end{proof} 

\begin{corollary} If $R$ is a nonzero ring and $S$ is an index set with $|S|\geqslant1$, then the cardinality of the polynomial ring $R[x_k: k\in S]$  is  $\aleph_{0}\times|R|\times|S|$.
\end{corollary}

\begin{proof} We know that $R[x_k: k\in S]$ is the monoid-ring $R[M]$ where $M=\bigoplus\limits_{k\in S}\omega$ and $\omega=\{0,1,2,\ldots\}$ is the additive monoid of natural numbers. Then by Lemma \ref{Lemma 3}, we have $|M|=\aleph_{0}\times|S|$.
Since $M$ is infinite and $R$ is nonzero, then by Theorem \ref{Theorem 1}, the cardinality of $R[M]$ is $|R|\times|M|=\aleph_{0}\times|R|\times|S|$.  
\end{proof}

In the rest of the article, if not stated, the rings involved are assumed to be commutative. 
 
\begin{lemma}\label{Lemma 5} If $R$ is an infinite ring and $S$ is a multiplicative set of non-zero-divisors of $R$, then $|R|=|S^{-1}R|$.
\end{lemma}

\begin{proof} The canonical map $R\rightarrow S^{-1}R$ given by $r\mapsto r/1$ is injective and hence $|R|\leqslant|S^{-1}R|$. The map $R\times S\rightarrow S^{-1}R$ given by $(r,s)\mapsto r/s$ is surjective. Then using the idempotency of infinite cardinals, we have
$|S^{-1}R|\leqslant|R\times S|=|R|\cdot|S|\leqslant|R|^{2}=|R|$. Hence, $|R|=|S^{-1}R|$.  
\end{proof}

\begin{lemma}\label{Lemma 6} Every zero-dimensional ring is canonically isomorphic to its total ring of fractions.  
\end{lemma}

\begin{proof} Let $R$ be a zero-dimensional ring. To prove the assertion, it suffices to show that 
$Z(R)=\bigcup\limits_{\mathfrak{p}
\in\Spec(R)}\mathfrak{p}$ where $Z(R)$ is the set of zero-divisors of $R$. For any ring $R$ we have $Z(R)\subseteq\bigcup\limits_{\mathfrak{p}
\in\Spec(R)}\mathfrak{p}$. It is also
well known that for any ring $R$ then $\bigcup\limits_{\mathfrak{p}\in\Min(R)}
\mathfrak{p}\subseteq Z(R)$ where $\Min(R)$ is the set of minimal prime ideals of $R$. By the hypothesis, $R$ is zero-dimensional and hence $\Min(R)=\Spec(R)$.
\end{proof}

\begin{theorem} Every commutative ring and its total ring of fractions have the same cardinality. 
\end{theorem}

\begin{proof} Let $R$ be a ring. If $R$ is infinite the assertion follows from Lemma \ref{Lemma 5}. But if $R$ is a finite ring then it is zero-dimensional and hence the assertion follows from Lemma \ref{Lemma 6}.
\end{proof}

The above result shows that every commutative ring is equinumerous (in bijection) with its total ring of fractions. This set-theoretic observation leads us to the following ring theory concept: 

\begin{definition} A commutative ring $R$ is called a \emph{balanced ring} if the canonical ring map $R\rightarrow T(R)$ is an isomorphism. 
\end{definition} 

For any ring $R$ we have $Z(R)\subseteq\bigcup\limits_{M\in\Max(R)}M$ where $\Max(R)$ is the set of maximal ideals of $R$. It can be seen that a ring $R$ is a balanced ring if and only if $Z(R)=\bigcup\limits_{M\in\Max(R)}M$, or equivalently, every non-zero-divisor element of $R$ is invertible in $R$. In Lemma \ref{Lemma 6} we observed that every zero-dimensional ring is a balanced ring. It can be easily seen that the total ring of fractions of any ring is a balanced ring. Every direct product of balanced rings is also a balanced ring. Next, we investigate further classes of rings that are balanced rings. 

\begin{lemma}\label{Theorem 2} If for a ring $R$ we have  $\Hom_{R}(R/M,R)\neq0$ for all $M\in \Max(R)$, then $R$ is a balanced ring. 
\end{lemma}

\begin{proof} It suffices to show that $M\subseteq Z(R)$. Take $a\in M$. By the hypothesis, we have a nonzero morphism of $R$-modules $f:R/M\rightarrow R$. Thus $b:=f(1+M)\neq0$ and $ab=0$. Hence, $a\in Z(R)$. 
\end{proof}

\begin{theorem} For a Noetherian ring $R$ the following assertions are equivalent: \\
$\mathbf{(i)}$ $R$ is a balanced ring. \\
$\mathbf{(ii)}$ $\Hom_{R}(R/M,R)\neq0$ for all $M\in \Max(R)$. \\
$\mathbf{(iii)}$ $\depth(R_{M})=0$ for all $M\in\Max(R)$.
\end{theorem}

\begin{proof} (i)$\Rightarrow$(ii): By hypothesis, $Z(R)=\bigcup\limits_{M\in\Max(R)}M$. Since $R$ is Noetherian, the set of associated primes of $R$ is a finite set and $Z(R)=\bigcup
\limits_{\mathfrak{p}\in\Ass(R)}\mathfrak{p}$. Then by the Prime Avoidance Lemma, $M\in \Ass(R)$. Thus $M=\Ann(x)$ for some nonzero $x\in R$. Then the map $R/M\rightarrow R$ given by $r+M\mapsto rx$ is a nonzero morphism of $R$-modules. \\
(ii)$\Rightarrow$(i): See Lemma \ref{Theorem 2}. \\
(i)$\Rightarrow$(iii): We observed that $M\in\Ass(R)$ for all $M\in \Max(R)$. Thus $MR_{M}\in\Ass(R_M)$. This shows that $\Hom_{R_{M}}(R_{M}/MR_{M},R_M)\neq0$. We know that a local ring $(S,\mathfrak{m})$ has zero depth if and only if $\Hom_{S}(S/\mathfrak{m},S)\neq0$. Thus $\depth(R_{M})=0$. \\
(iii)$\Rightarrow$(ii): We know that every finitely generated module over a Noetherian ring is finitely presented. 
It is also well known that for any modules $N$ and $N'$ over a ring $R$ with $N$ finitely presented and for any flat $R$-algebra $S$, we have the canonical isomorphism of $S$-modules $\Hom_R(N,N')\otimes_{R}S\simeq\Hom_{S}(N\otimes_{R}S,
N'\otimes_{R}S)$. Therefore, in our case we have $\Hom_R(R/M,R)\otimes_{R}R_M\simeq\Hom_{R_{M}}
(R_{M}/MR_{M},R_M)\neq0$. It follows that $\Hom_R(R/M,R)\neq0$. 
\end{proof}

\begin{theorem} Every self-injective ring is a balanced ring.
\end{theorem}

\begin{proof} Let $R$ be a self-injective ring (injective as a module over itself). To prove the assertion, it will be enough to show that every non-zero-divisor element $a\in R$ is invertible. 
The map $f:R\rightarrow R$ given by $r\mapsto ar$ is an injective morphism of $R$-modules. Then by hypothesis, there exists a morphism of $R$-modules $g:R\rightarrow R$ such that $gf$ is the identity map. It follows that $ag(1)=g(a)=g\big(f(1)\big)=1$.  
\end{proof}

\end{document}